\documentclass[a4paper,twoside]{article}
\usepackage{amssymb}
\usepackage{amsmath}
\usepackage{fancyhdr}
\usepackage{amsthm}
\usepackage{rotate}
\usepackage{graphicx}
\usepackage[T1]{fontenc}
\usepackage{afterpage}  %

\newtheorem{theorem}{Theorem}[section]
\newtheorem{lemma}[theorem]{Lemma}
\newtheorem{proposition}[theorem]{Proposition}
\newtheorem{corollary}[theorem]{Corollary}
\theoremstyle{definition}
\newtheorem{example}[theorem]{Example}
\newtheorem{definition}[theorem]{Definition}

\begin{document}

\afterpage{\rhead[]{\thepage} \chead[\small   I. I. Deriyenko and
W. A. Dudek]{\small $D$-loops} \lhead[\thepage]{} }                  %%%%%%%%% complete

\begin{center}
\vspace*{2pt} {\Large \textbf{$D$-loops}}\\[30pt]
{\large \textsf{\emph{Ivan I. Deriyenko \ and \ Wieslaw A. Dudek
}}}
\\[30pt]
\end{center}
{\bf Abstract.} \noindent{\footnotesize $D$-loops are loops with the antiautomorphic inverse property. The class of such loops is larger than the class of IP-loops. The smallest $D$-loops which is not an IP-loop has six elements. We prove several basic properties of such loops and present methods of constructions of $D$-loops from IP-loops. Unfortunately, a loop isotopic to a $D$-loop may not be a $D$-loop.}
\date{} 


\footnote{The main results of this paper were presented at the conference Loops'11 which was held 

\hspace*{3mm}in Trest, Czech Republic, July, 2011.}

\section*{\centerline{1. Introduction}}\setcounter{section}{1}\setcounter{theorem}{0}

A loop is a quasigroup $Q(\circ)$ with an identity element always denoted by $1$. A loop $Q(\circ)$ has the {\it inverse property}, i.e., it is an {\it IP-loop}, if for each its element $a$ there exists in $Q$ a uniquely determined {\it inverse element} $a'$ such that $a'\circ(a\circ b)=(b\circ a)\circ a'$. This means that in an {\it IP}-loop for {\it right} and {\it left translations}, i.e., for $R_a(x)=x\circ a$, $L_a(x)=a\circ
x$, we have
\begin{equation}\label{e1}
R_a^{-1}=R_{a'}, \ \ \ \ L_a^{-1}=L_{a'}.
\end{equation}
It is not difficult to shown that in an IP-loop $Q(\circ)$ for all $a,b\in Q$ hold
\begin{equation}\label{e2}
a\circ a'=a'\circ a=1, \ \  (a')'=a
\end{equation}
and
\begin{equation}\label{e3}
(a\circ b)'=b'\circ a'\,.
\end{equation}

On the other hand, in any loop $Q(\circ)$ for each $a\in Q$ there are uniquely determined {\it left} and {\it right loop-inverse}
elements $a_{_L}^{-1},a_{_R}^{-1}\in Q$ for which we have
$a_{_L}^{-1}\circ a=a\circ a_{_R}^{-1}\!=1$. A two-sided
loop-inverse element to $a\in Q$ is denoted by $a^{-1}.$ Clearly,
$(a^{-1})^{-1}=a$. Hence, an element $a^{-1}\in Q$ is loop-inverse
to $a\in Q$ if and only if $a\in Q$ is loop-inverse to $a^{-1}.$
In a loop each inverse element is loop-inverse but a loop-inverse
element may not be inverse.

\begin{example}\label{Ex-1} Consider the following loop $Q(\circ)$:
{\small{
$$
\begin{array}{c|ccccccccccc} \circ&1&2&3&4&5&6&7\\ \hline
 1&1&2&3&4&5&6&7\rule{0pt}{9pt}\\
 2&2&3&1&6&7&5&4\\
 3&3&1&2&7&6&4&5\\
 4&4&7&6&5&1&3&2\\
 5&5&6&7&1&4&2&3\\
 6&6&4&5&2&3&7&1\\
 7&7&5&4&3&2&1&6
\end{array}
$$}}
\indent
In this loop we have $a_{_L}^{-1}=a_{_R}^{-1}$ for each $a\in Q$.
Hence, each element of this loop is loop-inverse. But $a=5$ is not
an inverse element since $4\circ (5\circ 6)\ne 6$. The map
$h(x)=x_{_R}^{-1}$ is an antiautomorphism of this loop,
i.e., it satisfies the identity $h(x\circ y)=h(y)\circ h(x)$. \hfill$\Box{}$
\end{example}

Recall that a loop $Q(\circ)$ satisfies the {\it antiautomorphic inverse
property} if for each $x\in Q$ there exists a two-sided
loop-inverse element $x^{-1}$ such that $(x\circ
y)^{-1}=y^{-1}\circ x^{-1}$ holds for all $x,y\in Q$. A simple
example of such loop is an IP-loop. The above example proves
that there are also loops with this property which are not
IP-loops. Thus, the class of loops with this property is much
larger than the class of IP-loops.

This enables us to introduce the following definition.

\begin{definition}
A loop $Q(\circ)$ is called a {\it $D$-loop} if it satisfies the
{\it dual automorphic property} for $\varphi(x)=x_{_R}^{-1}$,
i.e., if
\begin{equation}\label{e4}
(x\circ y)_{_R}^{-1}=y_{_R}^{-1}\circ x_{_R}^{-1}
\end{equation}
holds for all $x,y\in Q$.
\end{definition}

\begin{theorem}\label{T1.3} A loop $Q(\circ)$ is a $D$-loop if and
only if it satisfies the identity
\begin{equation}\label{e5}
(x\circ y)_{_L}^{-1}=y_{_L}^{-1}\circ x_{_L}^{-1}.
\end{equation}
\end{theorem}
\begin{proof}
Suppose that $Q(\circ)$ is a $D$-loop. Since $x_{_L}^{-1}\circ
x=1$, from \eqref{e4} it follows
$$
1=1_{_R}^{-1}=(x_{_L}^{-1}\circ x)_{_R}^{-1}=x_{_R}^{-1}\circ
(x_{_L}^{-1})_{_R}^{-1},
$$
which together with $1=x_{_L}^{-1}\circ(x_{_L}^{-1})_{_R}^{-1}$
gives $x_{_L}^{-1}=x_{_R}^{-1}$. Thus \eqref{e4} implies
\eqref{e5}.

Analogously, using $x\circ x_{_R}^{-1}=1$ and
$(x_{_R}^{-1})_{_L}^{-1}\circ x_{_R}=1$ we can prove that
\eqref{e5} implies \eqref{e4}.
\end{proof}

\begin{corollary}\label{C1.4}
For all elements of $D$-loops we have
$$
\rule{35mm}{0mm}a_{_L}^{-1}=a_{_R}^{-1} \ \ \ and \ \ \
(a^{-1})^{-1}=a.\rule{35mm}{0mm}\Box{}
$$
\end{corollary}

This means that in the multiplication table of a $D$-loop $Q(\circ)$ its
neutral element is located symmetrically with respect to the main
diagonal and the class of all $D$-loops coincides with the class
of loops with the antiautomorphic inverse property but we'll keep the
term {\it $D$-loop} since it is shorter and more convenient to
use.

\section*{\centerline{2. Constructions of D-loops}}\setcounter{section}{2}\setcounter{theorem}{0}

Below we present several methods of verification when a given loop
is a $D$-loop. To describe these methods we must reminder some
definitions from \cite{Bel1}, \cite{D'91} and \cite{D'08}.

\begin{definition}
Let $Q(\cdot)$ be a loop. A permutation $\varphi_a$ of $Q$, where
$a\in Q$, is called a {\it right middle translation} or a {\it right track} (shortly: {\it track}) of $Q(\cdot)$ if
\begin{equation}\label{e6}
x\cdot\varphi_a(x)=a
\end{equation}
holds for all $x\in Q$. By a {\it left middle translation} or a {\it left track} we mean a
permutation $\lambda_a$ such that
\begin{equation}\label{e7}
\lambda_a(x)\cdot x=a .
\end{equation}
\end{definition}

It is clear that $\lambda_a=\varphi_a^{-1}$ and
$\varphi_1(x)=x_R^{-1}$ for all $a,x\in Q$.

\medskip
The permutation $\varphi_a$ selects in the multiplication table of
a given loop the number of columns in an element $a$ appears. For the loop defined in Example \ref{Ex-1}
$$
\varphi_4=\left(\arraycolsep=.7mm\begin{array}{ccccccccccc} 1&2&3&4&5&6&7\\
4&7&6&1&5&2&3\end{array}\right)=(1\,4)(2\,7\,3\,6)(5).$$ Further,
permutations will be written in the form of cycles, cycles will be
separated by dots. For example, the above permutation $\varphi_4$
will be written as $\varphi_4=(1\,4.\,2\,7\,3\,6.\,5.)$.

It is clear that a loop $Q(\cdot)$, where $Q=\{1,2,\ldots,n\}$,
can be identified with the set
$\{\varphi_1,\varphi_2,\ldots,\varphi_n,\}$ of its tracks.

\begin{theorem}\label{T2.2}                                                    %T2
A loop $Q(\cdot)$ is a $D$-loop if and only if
\begin{equation}\label{e8}
\varphi_1\varphi_a\varphi_1=\varphi_{a^{-1}}^{-1}
\end{equation}
for every $a\in Q$, where $a^{-1}$ is $($right$)$ inverse to $a$.
\end{theorem}
\begin{proof}
Let $Q(\cdot)$ be a D-loop. Then $x_R^{-1}=x^{-1}$ for every $x\in
Q$ and, according to \eqref{e6}, for all $a,x\in Q$ we have
$\varphi_a^{-1}(x)\cdot x=a$. Hence
$$
a^{-1}=(\varphi_a^{-1}(x)\cdot x)^{-1}=x^{-1}\cdot
(\varphi_a^{-1}(x))^{-1}=
\varphi_1(x)\cdot\varphi_1\varphi_a^{-1}(x).
$$
Since also $a^{-1}=\varphi_1(x)\cdot\varphi_{a^{-1}}\varphi_1(x)$,
from the above we obtain
$\varphi_1\varphi_a^{-1}=\varphi_{a^{-1}}\varphi_1$, which implies
\eqref{e8}.

Conversely, let $x\cdot y=a$. Then $y=\varphi_a(x)$. Hence
\[\arraycolsep=.5mm\begin{array}{rl}
y_R^{-1}\cdot x_R^{-1}&
=\varphi_1\varphi_a(x)\cdot\varphi_1(x)=\varphi_1\varphi_a\varphi_1(x^{-1})\cdot
x^{-1}\\[4pt]
&\stackrel{\eqref{e8}}{=}\varphi_{a^{-1}}^{-1}(x^{-1})\cdot
x^{-1}=\lambda_{a^{-1}}(x^{-1})\cdot
x^{-1}\stackrel{\eqref{e7}}{=}a_R^{-1}=(x\cdot y)_R^{-1}.
\end{array}
\]
This completes the proof.
\end{proof}

\begin{corollary}\label{C2.3}
A loop $Q(\cdot)$ is a $D$-loop if and only if it satisfies one of
the following identities:

\medskip$(a)$ \ $\varphi_1\varphi_a^{-1}\varphi_1=\varphi_{a^{-1}}$,

\medskip$(b)$ \ $\varphi_1 R_a\varphi_1=L_{a^{-1}}$,

\medskip$(c)$ \ $\varphi_1L_a\varphi_1=R_{a^{-1}}$.
\end{corollary}
\begin{proof}
Indeed, \eqref{e8} can be written in the form
$\varphi_1\varphi_{a^{-1}}\varphi_1=\varphi_a^{-1}$, which, in
view of $\varphi_1\varphi_1=id_Q$, is equivalent to $(a)$.
Moreover, $(x\cdot a)^{-1}=a^{-1}\cdot x^{-1}$ means that
$\varphi_1R_a=L_{a^{-1}}\varphi_1$. The last is equivalent to
$(b)$ and $(c)$.
\end{proof}

\begin{example}\label{Ex-2}
Consider the loop $Q(\cdot)$:
{\small{
$$\begin{array}{c|cccccccc}
 \cdot&1&2&3&4&5&6\\ \hline
 1&1&2&3&4&5&6\rule{0pt}{10pt}\\
 2&2&1&4&3&6&5\\
 3&3&5&1&6&4&2\\
 4&4&6&5&2&1&3\\
 5&5&3&6&1&2&4\\
 6&6&4&2&5&3&1
\end{array}$$}}

We will use Theorem \ref{T2.2} to verify that this loop is a
$D$-loop.

We have
\[\begin{array}{lllllll}
\varphi_1=(1.\,2.\,3.\,4\,5.\,6.)&\rule{10mm}{0mm}&\varphi_4=(1\,4.\,2\,3\,5\,6.)\\[2pt]
\varphi_2=(1\,2.\,3\,6.\,4.\,5.)&&\varphi_5=(1\,5.\,2\,6\,4\,3.)\\[2pt]
\varphi_3=(1\,3.\,2\,4\,6\,5.)&&\varphi_6=(1\,6.\,2\,5\,3\,4.)
\end{array}
\]
We have to check the condition \eqref{e8} for $a=2,3,4,5,6$
because $\varphi_1\varphi_1\varphi_1=\varphi_1^{-1}$ holds in each
loop. Permutations $\varphi_1$ and $\varphi_2$ have disjoint
cycles hence
$\varphi_1\varphi_2\varphi_1=\varphi_2=\varphi_2^{-1}$. In other
cases we obtain:
\[\begin{array}{lllllll}
\varphi_1\varphi_3\varphi_1=(1\,3.\,2\,5\,6\,4.)=\varphi_3^{-1}&\rule{8mm}{0mm}
&\varphi_1\varphi_5\varphi_1=(1\,4.\,2\,6\,5\,3.)=\varphi_4^{-1}\\[2pt]
\varphi_1\varphi_4\varphi_1=(1\,5.\,2\,3\,4\,6.)=\varphi_5^{-1}&\rule{8mm}{0mm}
&\varphi_1\varphi_6\varphi_1=(1\,6.\,2\,4\,3\,5.)=\varphi_6^{-1}
\end{array}
\]
This shows that $Q(\circ)$ is a $D$-loop. \hfill$\Box{}$
\end{example}

Note that in general loops isotopic to $D$-loops are not $D$-loops.

\begin{example}\label{Ex-3}
The following loop
{\small{$$
\begin{array}{c|ccccccccccc}
 \circ&1&2&3&4&5&6\\ \hline
 1&1&2&3&4&5&6\rule{0pt}{10pt}\\
 2&2&1&6&3&4&5\\
 3&3&6&2&5&1&4\\
 4&4&5&1&6&2&3\\
 5&5&3&4&1&6&2\\
 6&6&4&5&2&3&1
\end{array}
$$}}\noindent
is isotopic to a $D$-loop $Q(\circ)$ from the previous example.
This isotopy has the form $\gamma(x\circ y)=
\alpha(x)\cdot\beta(y)$, where $\alpha=(1\,4\,2.\,3.\,5.\,6.)$,
$\beta=(1\,2\,5\,4\,6.\,3.)$, $\gamma=(1\,6\,4\,3\,5\,2.)$. The
loop $Q(\circ)$ is not a $D$-loop since $3_{_L}^{-1}\ne
3_{_R}^{-1}.$ \hfill$\Box{}$
\end{example}

\begin{theorem}\label{T2.6}
Let $Q(\cdot)$ be an IP-loop and let $a\in Q$ be fixed. If an element $a'\in Q$ is inverse to $a$ in $Q(\cdot)$, then $Q(\circ)$ with the operation
\begin{equation}\label{e9}
x\circ y=R_{a'}(x)\cdot L_a(y)
\end{equation}
is a D-loop with the same identity as in $Q(\cdot)$.
\end{theorem}
\begin{proof}
It is clear that $Q(\circ)$ is a quasigroup. Let an element $a'\in
Q$ be inverse to $a$ in $Q(\cdot)$. Then
$$
x\circ 1=R_{a'}(x)\cdot L_a(1)=(x\cdot a')\cdot a=x.
$$
Similarly $1\circ x=x$. Hence $Q(\circ)$ is a loop with the same
identity as in $Q(\cdot)$.

Moreover, for every $x\in Q$ there exists $\overline{x}\in Q$ such
that
$$
1=x\circ\overline{x}=R_{a'}(x)\cdot L_a(\overline{x})=(x\cdot
a')\cdot(a\cdot\overline{x}),
$$
which gives $a\cdot\overline{x}=(x\cdot a')^{-1}=(a')^{-1}\cdot
x^{-1}=a\cdot x^{-1}$. Thus $\overline{x}=x^{-1}$ for every $x\in
Q$. Hence
\[\arraycolsep=.5mm
\begin{array}{rl}
(x\circ y)^{-1}&=\left(R_{a'}(x)\cdot
L_a(y)\right)^{-1}=\left((x\cdot a')\cdot (a\cdot
y)\right)^{-1}\\[2mm]
&=(a\cdot y)^{-1}\cdot (x\cdot a')^{-1}=(y^{-1}\cdot
a^{-1})\cdot((a')^{-1}\cdot x^{-1})\\[2mm]
&=(y^{-1}\cdot a')\cdot (a\cdot x^{-1})=R_{a'}(y^{-1})\cdot
L_a(x^{-1})=y^{-1}\circ x^{-1}.
\end{array}
\]
Therefore $Q(\circ)$ is a $D$-loop.
\end{proof}

\begin{corollary}
Any IP-loop of order $n$ determines\ $n-1$ isotopic $D$-loops.
\end{corollary}

\begin{example}\label{Ex-4} Starting from the following {\it
IP}-loop:
{\small{
$$
\begin{array}{c|ccccccccccc}
\cdot&1&2&3&4&5&6&7\\ \hline
 1&1&2&3&4&5&6&7\rule{0pt}{10pt}\\
 2&2&3&1&6&7&5&4\\
 3&3&1&2&7&6&4&5\\
 4&4&7&6&5&1&2&3\\
 5&5&6&7&1&4&3&2\\
 6&6&4&5&3&2&7&1\\
 7&7&5&4&2&3&1&6
\end{array}
$$}}\noindent
and using \eqref{e9} with $a=2$ we obtain a $D$-loop:
{\small{
$$
\begin{array}{c|ccccccccccc}
\circ&1&2&3&4&5&6&7\\ \hline
 1&1&2&3&4&5&6&7\rule{0pt}{10pt}\\
 2&2&3&1&6&7&5&4\\
 3&3&1&2&5&4&7&6\\
 4&4&5&6&7&1&2&3\\
 5&5&4&7&1&6&3&2\\
 6&6&7&5&3&2&4&1\\
 7&7&6&4&2&3&1&5
\end{array}
$$}}\noindent
which is not an {\it IP}-loop because $3\circ (2\circ 5)\ne 5$.
Hence $a=2$ is not an inverse element in $Q(\circ)$. Putting
$x*y=R_3(x)\circ L_2(y)$ we obtain a quasigroup:
{\small{
$$
\begin{array}{c|ccccccccccc}
 *&1&2&3&4&5&6&7\\ \hline
 1&1&2&3&7&6&4&5\rule{0pt}{10pt}\\
 2&2&3&1&6&7&5&4\\
 3&3&1&2&5&4&7&6\\
 4&7&5&6&4&1&2&3\\
 5&6&4&7&1&5&3&2\\
 6&4&7&5&3&2&6&1\\
 7&5&6&4&2&3&1&7
\end{array}
$$}}\noindent
which is isotopic to the initial $D$-loop $Q(\circ)$ but it is not
a $D$-loop. This means that in Theorem \ref{T2.6} the assumption
on $a$ can not be ignored.\hfill$\Box{}$
\end{example}

\begin{proposition}\label{P2.9}
An element $a\in Q$ used in Theorem $\ref{T2.6}$ has the same inverse element in
$Q(\cdot)$ and $Q(\circ)$ defined by \eqref{e9} if and only if
\begin{equation}\label{e10}
x\cdot a=x\circ a \ \ \ and \ \ \ a'\cdot x=a'\circ x
\end{equation}
for all $x\in Q$.
\end{proposition}
\begin{proof}
Let $a'\in Q$ be inverse to $a$ in $Q(\cdot)$ and $Q(\circ)$. Then $a=(a')'$ and
$$
z=(z\circ a')\circ a=\left(R_{a'}(z)\cdot L_a(a')\right)\circ
a=R_{a'}(z)\circ a ,
$$
which for $z=x\cdot a$ gives $x\cdot a=x\circ a$.

Similarly,
$$
z=a'\circ (a\circ z)=a'\circ(R_{a'}(a)\cdot L_a(z))=a'\circ L_a(z)
$$
for $z=a'\cdot x$ implies $a'\cdot x=a'\circ x$.

Conversely, if $a'\in Q$ is inverse to $a$ in $Q(\cdot)$ and \eqref{e10} are satisfied, then
$$
(x\circ a)\circ a'=(x\cdot a)\circ a'=R_{a'}(x\cdot a)\cdot
L_a(a')= ((x\cdot a)\cdot a')\cdot (a\cdot a')=x.
$$
Analogously $a'\circ(a\circ x)=a'\cdot(a\circ x)=x$. Hence $a'$ is
inverse to $a$ in $Q(\circ)$.
\end{proof}

\begin{corollary}
An element $a\in Q$ used in Theorem $\ref{T2.6}$ has the same inverse element in $Q(\cdot)$ and $Q(\circ)$ defined by \eqref{e9} if and only if the
multiplication tables of these two loops have the same
$a$--columns and the same $a'$--rows.
\end{corollary}

\begin{proposition}\label{P2.11}
An element $a\in Q$ used in Theorem $\ref{T2.6}$ has the same inverse element in $Q(\cdot)$ and $Q(\circ)$ defined by \eqref{e9} if and only if
\begin{equation}\label{e11}
L_aL_a=L_{a^2}\ \ \ and \ \ \ R_aR_a=R_{a^2}
\end{equation}
where $L_a$ and $R_a$ are translations in $Q(\cdot)$.
\end{proposition}
\begin{proof}
Let $a\in Q$ has the same inverse $a'$ in $Q(\cdot)$ and $Q(\circ)$. Then for every $x\in Q$ we have
\[\arraycolsep=.5mm\begin{array}{lll}
x&=a'\circ(a\circ x)=R_{a'}(a')\cdot L_a(R_{a'}(a)\cdot L_a(x))\\[2pt]
&=R_{a'}(a')L_aL_a(x)=L_{a'\cdot a'}L_aL_a(x)=L_{(a^2)'}L_aL_a(x),
\end{array}
\]
whence, applying \eqref{e1}, we get $L_aL_a=L_{(a^2)'}^{-1}=L_{a^2}$.

Similarly, for every $z\in Q$ we have
$$
z\circ a=R_{a'}(z)\cdot L_a(a)=R_{a^2}R_{a'}(z) ,
$$
which for $z=R_a(x)$, by \eqref{e1}, gives
$$
R_a(x)\circ a=R_{a^2}R_{a'}R_a(x)=R_{a^2}(x).
$$
Hence $R_aR_a=R_{a^2}$. This proves \eqref{e11}.

The converse statement is obvious.
\end{proof}

Below we present a simple method of construction of new loops from given loops. This method is based on {\it exchange of tracks}. Next, this method will be applied to the construction of $D$-loops.

\medskip

Let $\{\varphi_1,\varphi_2,\ldots,\varphi_n\}$ be tracks of a $D$-loop $Q(\cdot)$ with the identity $1$. We say that for $i\ne j\ne 1$ tracks $\varphi_i$, $\varphi_j$ are {\it decomposable} if there exist two nonempty subsets $X,Y$ of $Q$ such that $Q=X\cup Y,$ $X\cap Y=\emptyset$, $1\in X$ and
\begin{equation}\label{e12}
\left\{\begin{array}{l}\varphi_i=\bar{\varphi}_i\hat{\varphi}_i\\
\varphi_j=\bar{\varphi}_j\hat{\varphi}_j
\end{array}\right.
\end{equation}
where $\bar{\,\varphi}_i,\bar{\,\varphi}_j$ are permutations of $X$, $\hat{\varphi}_i,\hat{\varphi}_j$ are permutations of $Y.$

Putting
\begin{equation}\label{e13}
\left\{\begin{array}{l}
\psi_i=\bar{\varphi}_i\hat{\varphi}_j\\
\psi_j=\bar{\varphi}_j\hat{\varphi}_i
\end{array}\right.
\end{equation}
and $\psi_k=\varphi_k$ for $k\not\in\{i,j\}$ we obtain the new system of tracks which defines on $Q$ the new loop $Q(\circ)$ with the same identity as in $Q(\cdot)$.

\begin{example}\label{Ex-5}The loop $Q(\cdot)$ defined by
{\small{$$
\begin{array}{c|ccccccccccccccc}
 \cdot&1&2&3&4&5&6&7&8\\ \hline
 1&1&2&3&4&5&6&7&8\rule{0pt}{10pt}\\
 2&2&3&4&1&6&7&8&5\\
 3&3&4&1&2&7&8&5&6\\
 4&4&1&2&3&8&5&6&7\\
 5&5&8&7&6&1&4&3&2\\
 6&6&5&8&7&2&1&4&3\\
 7&7&6&5&8&3&2&1&4\\
 8&8&7&6&5&4&3&2&1
\end{array}
$$}}\noindent
is a group (so, it is a $D$-loop) with the following tracks:
$$
\begin{array}{lllllll}
\varphi_1=(1.24.3.4.5.6.7.8.)&&\varphi_2=(12.34.5876.)&&\varphi_3=(13.2.4.57.68.)\\
\varphi_4=(14.23.5678.)&&\varphi_5=(15.37.2846.)&&\varphi_6=(16.38.2547.)\\
\varphi_7=(17.35.2648.)&&\varphi_8=(18.36.2745.)&&
\end{array}
$$
For $(i,j)\in\{(2,3),(2,4),(3,4),(5,7),(6,8)\}$ tracks $\varphi_i,\varphi_j$ are decomposable.
In the case $(i,j)=(6,8)$ we have
$$
\left\{\begin{array}{l}
\varphi_6=\bar{\varphi}_6\hat{\varphi}_6, \ \ {\rm where} \ \ \bar{\varphi}_6=(16.38.), \ \hat{\varphi}_6=(2547.)\\
\varphi_8=\bar{\varphi}_8\hat{\varphi}_8,\ \ {\rm where} \ \ \bar{\varphi}_8=(18.36.), \ \hat{\varphi}_8=(2745.)
\end{array}\right.
$$
whence, according to \eqref{e13}, we obtain
$$
\left\{\begin{array}{l}\psi_6=\bar{\varphi}_6\hat{\varphi}_8=(16.38.2745.)\\ \psi_8=\bar{\varphi}_8\hat{\varphi}_6=(18.36.2547.)
\end{array}\right.
$$
and $\psi_k=\varphi_k$ for $k=1,2,3,4,5,7$.

This new system of tracks $\{\psi_1,\psi_2,\ldots,\psi_8\}$ defines the loop $Q(\circ)$:
{\small{$$
\begin{array}{c|ccccccccccccccc}
 \circ&1&2&3&4&5&6&7&8\\ \hline
 1&1&2&3&4&5&6&7&8\rule{0pt}{10pt}\\
 2&2&3&4&1&\frame{\,8\,\rule{0pt}{8pt}}&7&\frame{\,6\,\rule{0pt}{8pt}}&5\\
 3&3&4&1&2&7&8&5&6\\
 4&4&1&2&3&\frame{\,6\,\rule{0pt}{8pt}}&5&\frame{\,8\,\rule{0pt}{8pt}}&7\\
 5&5&\frame{\,6\,\rule{0pt}{8pt}}&7&\frame{\,8\,\rule{0pt}{8pt}}&1&4&3&2\\
 6&6&5&8&7&2&1&4&3\\
 7&7&\frame{\,8\,\rule{0pt}{8pt}}&5&\frame{\,6\,\rule{0pt}{8pt}}&3&2&1&4\\
 8&8&7&6&5&4&3&2&1
\end{array}
$$}}\noindent
where items changed by tracks $\psi_6$ and $\psi_8$ are entered in the box.
\hfill$\Box{}$
\end{example}

This new loop $Q(\circ)$ can be used for the construction of
another loop since it has the same pair of decomposable tracks as
$Q(\cdot)$. So, for the construction of new loops we can use not
only one but also two or more pairs of decomposable tracks. Using
different pairs of decomposable tracks we obtain different loops
which may not be isotopic. Obtained loops may not be isotopic to
the initial loop $Q(\cdot)$, too.

\begin{example}\label{Ex-6}
Direct computations show that this loop
{\small{
$$
\begin{array}{c|ccccccccccc}
\cdot&1&2&3&4&5&6&7&8\\ \hline                                                                                 1&1&2&3&4&5&6&7&8\\
2&2&1&5&7&3&8&4&6\\                                                                                            3&3&8&6&1&4&2&5&7\\                                                                                            4&4&6&1&5&7&3&8&2\\                                                                                            5&5&7&4&2&8&1&6&3\\                                                                                            6&6&4&8&3&1&7&2&5\\                                                                                            7&7&5&2&8&6&4&3&1\\                                                                                            8&8&3&7&6&2&5&1&4
\end{array}$$}}\noindent
is a $D$-loop. It hasn't got any decomposable pair of tracks.
\hfill$\Box{}$
\end{example}

\begin{theorem}\label{T2.14}
Let $Q(\cdot)$ be a $D$-loop with the identity $1$. If $\varphi_i,\varphi_j$, where $i\cdot j=1$ and $i\ne j$, are decomposable tracks of $Q(\cdot)$, then a loop $Q(\circ)$ obtained from $Q(\cdot)$ by exchange of tracks is a $D$-loop.
\end{theorem}\begin{proof} Since $Q(\circ)$ is a loop it is sufficient to show that $\psi_1\psi_k\psi_1=\psi_{k^{-1}}^{-1}$ for every $k\in Q$ (Theorem \ref{T2.2}). For $k\not\in\{i,j\}$ we have $\psi_k=\varphi_k$, so for $k\not\in\{i,j\}$ this condition is satisfied by the assumption. For $k=i$ we have
$$
\psi_1\psi_i\psi_1=\varphi_1\bar{\varphi}_i\hat{\varphi}_j\varphi_1=
(\varphi_1\bar{\varphi}_i\varphi_1)(\varphi_1\hat{\varphi}_j\varphi_1)=
\bar{\varphi_j}^{-1}\hat{\varphi_i}^{-1}=\psi_j^{-1}=\psi_{i^{-1}}^{-1}
$$
because $\varphi_1^2=\varepsilon$,\ $\bar{\varphi}_i\hat{\varphi}_j=\hat{\varphi}_j\bar{\varphi}_i$ and $i\cdot j=1.$

For $k=j$ the proof is analogous. So, $Q(\circ)$ is a $D$-loop.
\end{proof}

The assumption $i\cdot j=1$ is essential. Indeed, in Example \ref{Ex-5} tracks $\varphi_3$, $\varphi_4$ are decomposable, $4\cdot 3\ne 1$, $4^{-1}=2$ and $\psi_1\psi_4\psi_1=\psi_4\ne\psi_2^{-1}$.  So, a loop determined by tracks $\psi_1,\ldots,\psi_8$ is not a $D$-loop.

The $D$-loop $Q(\circ)$ constructed in Example \ref{Ex-5} is not isotopic to the initial group $Q(\cdot)$ since  $(7\circ7)\circ 2\ne 7\circ (7\circ 2)$. In this loop we also have $7\circ(7\circ 2)\ne 2$, so it is not an IP-loop, too.

 \begin{example}\label{Ex-5a}The loop $Q(\cdot)$ defined by
{\small{
$$
\begin{array}{c|ccccccccccc}
\cdot&1&2&3&4&5&6&7&8\\ \hline
1&1&2&3&4&5&6&7&8\\
2&2&1&5&7&3&8&4&6\\
3&3&8&4&1&7&2&6&5\\
4&4&6&1&3&8&7&5&2\\
5&5&7&8&2&6&1&3&4\\
6&6&4&7&8&1&5&2&3\\
7&7&5&2&6&4&3&8&1\\
8&8&3&6&5&2&4&1&7
\end{array}$$}}\noindent
is a $ D$-loop with the following tracks:
$$
\begin{array}{lllllll}
\varphi_1=(1.2.34.56.78.)&&\varphi_2=(12.367.485.)&&\varphi_3=(13.4.25768.)\\
\varphi_4=(14.3.27586.)&&\varphi_5=(15.6.23847.)&&\varphi_6=(16.5.28374.)\\
\varphi_7=(17.8.24635.)&&\varphi_8=(18.7.26453.)&&
\end{array}
$$
For $(i,j)\in\{(3,4),(5,6),(7,8)\}$ tracks $\varphi_i,\varphi_j$ are decomposable. Each pair of such tracks gives a $D$-loop. Obtained loops are isotopic  but they are not isotopic to $Q(\cdot)$ since they and $Q(\cdot)$ have different indicators $\Phi^*$. (Isotopic loops have the same indicators -- see \cite{DD}.)

If for the construction of a new loop we use two pairs of decomposable tracks:  $\varphi_3,\varphi_4$ and $\varphi_5,\varphi_6$, or $\varphi_3,\varphi_4$ and $\varphi_7,\varphi_8$, or $\varphi_5,\varphi_6$ and $\varphi_7,\varphi_8$, then we obtain three isotopic $D$-loops. These loops are not isotopic either to $Q(\cdot)$ or to the previous because have different indicators $\Phi^*$.

Also in the case when we use three pairs of decomposable tracks obtained $D$-loop. It is not isotopic to any of the previous.

So, from this $D$-loop we obtain three nonisotopic $D$-loops which also are not isotopic to the initial $D$-loop $Q(\cdot)$.\hfill\qed
\end{example}

\medskip

As it is well known with each quasigroup $Q(\cdot)$ we can conjugate five new quasigroups (called {\it parastrophes} of $Q(\cdot)$) by permuting the variables in the defining equation. Namely, if $Q_0=Q(\cdot)$ is a fixed quasigroup, then its parastrophes have the form
$$
\begin{array}{llcclllccll}
Q(\backslash)&&x\backslash z=y&\Longleftrightarrow& x\cdot y=z,\\[2pt]
Q(/)&&z/y=x&\Longleftrightarrow& x\cdot y=z,\\[2pt]
Q(*)&&y*x=z&\Longleftrightarrow& x\cdot y=z,\\[2pt]
Q(\bullet)&&y\bullet z=x&\Longleftrightarrow& x\cdot y=z,\\[2pt]
Q(\triangleleft)&&z\triangleleft x=y&\Longleftrightarrow& x\cdot y=z.
\end{array}
$$

\begin{theorem} Parastrophes of a $D$-loop $Q(\cdot)$ are isomorphic to one of the following quasigroups: $Q(\cdot)$, $Q(\backslash)$, $Q(/)$.
\end{theorem}
\begin{proof}
Indeed, if $Q(\cdot)$ is a $D$-loop, $\varphi_1$ -- its track determined by the identity of $Q(\cdot)$, then, according to the definition of $D$-loops, we have
$$\varphi_1(x\cdot y)=\varphi_1(y)\cdot\varphi_1(x).
$$
Hence
$$
\varphi_1(y*x)=\varphi_1(z)\Longleftrightarrow \varphi_1(x\cdot y)=\varphi_1(z)
\Longleftrightarrow \varphi_1(y)\cdot\varphi_1(x)=\varphi_1(z).
$$
So, $\varphi_1(y*x)=\varphi_1(y)\cdot\varphi_1(x),$ i.e., $Q(*)$ and $Q(\cdot)$ are isomorphic.

\medskip
Further,
$$\begin{array}{rlll}
\varphi_1(y\bullet z)=\varphi_1(x)&\Longleftrightarrow \varphi_1(x\cdot y)=\varphi_1(z)
\Longleftrightarrow \varphi_1(y)\cdot\varphi_1(x)=\varphi_1(z)\\[4pt]
&\Longleftrightarrow \varphi_1(y)\backslash\varphi_1(z)=\varphi_1(x).
\end{array}$$
Thus, $\varphi_1(y\bullet z)=\varphi_1(y\backslash z)$. Consequently, $Q(\bullet)\cong Q(\backslash)$.

\medskip

Analogously,
$$\begin{array}{rlll}
\varphi_1(z\triangleleft x)=\varphi_1(y)&\Longleftrightarrow \varphi_1(x\cdot y)=\varphi_1(z)
\Longleftrightarrow \varphi_1(y)\cdot\varphi_1(x)=\varphi_1(z)\\[4pt]
&\Longleftrightarrow \varphi_1(z)/\varphi_1(x)=\varphi_1(y),
\end{array}$$
whence $\varphi_1(z\triangleleft x)=\varphi_1(z/x)$. So, $Q(\triangleleft)\cong Q(/)$.
\end{proof}

\section*{\centerline{3. Loops isotopic to D-loops}}\setcounter{section}{3}\setcounter{theorem}{0}

As was mentioned earlier, loops isotopic to $D$-loops are not $D$-loops in general, but in some cases principal isotopes of $D$-loops are $D$-loops. Below we find conditions under which $D$-loops are isotopic to groups and conditions under which a principal isotope of a $D$-loop is a $D$-loop.
\begin{definition}
Let $Q(\cdot)$, where $Q=\{1,2,\ldots,n\}$, be a quasigroup. By a
{\it spin} of a quasigroup $Q(\cdot)$ we mean a permutation
$$
\varphi_{ij}=\varphi_i\varphi_j^{-1}=\varphi_i\lambda_j\, ,
$$
where $\varphi_i$ and $\lambda_j$ are right and left tracks of
$Q(\cdot)$ respectively.
\end{definition}

Obviously $\varphi_{ii}=\varepsilon$ for $i\in Q$ and
$\varphi_{ij}\ne\varphi_{ik}$ for $j\ne k$, but the situation
where $\varphi_{ij}=\varphi_{kl}$ for some $i,j,k,l\in Q$ also is
possible (cf. \cite{D'08}). Hence the collection $\Phi$ of all
spins of a given quasigroup $Q(\cdot)$ can be divided into
disjoint subsets $\Phi_i=\{\varphi_{ij}:j\in Q\}$ (called {\it
spin-basis}) in which all elements are different. Generally,
$\Phi_i$ are not closed under the composition of permutations but
in some cases $\Phi_i$ are groups.

In \cite{D'08} the following result is proved.

\begin{theorem}\label{T3.2}
A quasigroup $Q(\cdot)$ is isotopic to some group if and only if
its spin-basis $\Phi_1$ is a group.\hfill$\Box{}$
\end{theorem}

In this case $\Phi_1=\Phi_i$ for all $i\in Q$.

\begin{theorem}\label{T3.3}
In $D$-loops we have $\Phi=\langle\Phi_1\rangle
=\{\varphi_{1i}\varphi_{1j}:i,j\in Q\}$.
\end{theorem}
\begin{proof}
Indeed, by Corollary \ref{C2.3}
$$
\varphi_{1i}\varphi_{1j}=\varphi_1\varphi_i^{-1}\varphi_1\varphi_j^{-1}=
(\varphi_1\varphi_i^{-1}\varphi_1)\varphi_j^{-1}=\varphi_{i^{-1}}\varphi_j^{-1}=\varphi_{i^{-1}j}\in
\Phi
$$
and conversely, each $\varphi_{ij}\in \Phi$ can be written
in the form $\varphi_{ij} =\varphi_{1i^{-1}}\varphi_{1j}$.
\end{proof}

\begin{corollary}\label{C3.4}
A $D$-loop is isotopic to a group if and only if
$\langle\Phi_1\rangle=\Phi_1$.
\end{corollary}
\begin{proof}
If a $D$-loop $Q(\cdot)$ is isotopic to a group, then, by Theorem
\ref{T3.2}, $\Phi_1$ is a group. Hence $\langle\Phi_1\rangle=\Phi_1$.

Conversely, if $\langle\Phi_1\rangle=\Phi_1$, then
$\varphi_{1i}\varphi_{i1}=\varphi_{11}$ implies
$\varphi_{1i}^{-1}=\varphi_{i1}\in
\Phi=\langle\Phi_1\rangle=\Phi_1$ which means that $\Phi_1$ is a
group. Thus $Q(\circ)$ is isotopic to some group.
\end{proof}

\begin{corollary}\label{C3.5}
A $D$-loop is isotopic to a group if and only if
$\Phi_1$ is closed under a composition of permutations.
\end{corollary}
\begin{proof}
If a $D$-loop $Q(\cdot)$ is isotopic to a group, then, by Theorem
\ref{T3.2}, $\Phi_1$ is a group. Hence $\Phi_1$ is closed under
a composition of permutations.

Conversely, if $\Phi_1$ is closed under a composition of
permutations, then, in view of Theorem \ref{T3.3}, from
$\varphi_{1i}\varphi_{i1}=\varphi_{11}$ it follows
$\varphi_{1i}^{-1}=\varphi_{i1}\in
\Phi=\langle\Phi_1\rangle=\Phi_1$, which means that $\Phi_1$ is a
group. Thus $Q(\cdot)$ is isotopic to some group.
\end{proof}

\begin{corollary}\label{C3.6}
A $D$-loop $Q(\cdot)$  is isotopic to a group if and only if for
all $i,j\in Q$ there exists $k\in Q$ such that $\varphi_i\varphi_1\varphi_j=\varphi_k$.
\end{corollary}
\begin{proof}Indeed, $\varphi_{1j}\varphi_{1i}=\varphi_{1k}$ means that
$\varphi_1\varphi_j^{-1}\varphi_1\varphi_i^{-1}=\varphi_1\varphi_k^{-1}$.
Thus $\varphi_j^{-1}\varphi_1\varphi_i^{-1}=\varphi_k^{-1}$. Hence
$\varphi_k=\varphi_i\varphi_1\varphi_j$.
\end{proof}

\begin{theorem}\label{T3.7}
If a quasigroup $Q(\cdot)$ is isotopic to a $D$-loop $Q(\circ)$,
then there exists a permutation $\sigma$ of $Q$ and an element
$p\in Q$ and such that
\begin{equation}\label{e14}
\varphi_p\varphi_i^{-1}\varphi_p=\varphi_{\sigma(i)}
\end{equation}
for all tracks $\varphi_i$ of $Q(\cdot)$.
\end{theorem}
\begin{proof}
Let a quasigroup $Q(\cdot)$ be isotopic to a $D$-loop $Q(\circ)$.
Then
\begin{equation}\label{e15}
\gamma(x\cdot y)=\alpha(x)\circ\beta(y)
 \end{equation}
for some permutations $\alpha,\beta,\gamma$ of $Q$. Thus for all
$i,x\in Q$ we have
$$
\gamma(i)=\gamma(x\cdot\varphi_i(x))=\alpha(x)\circ\beta\varphi_i(x),
$$
where $\varphi_i$ is a right track of $Q(\cdot)$. Hence
$$
\gamma(i)=z\circ\beta\varphi_i\alpha^{-1}(z)
$$
for all $i\in Q$ and $z=\alpha(x)$. This together with
$\gamma(i)=z\circ\psi_{\gamma(i)}(z)$ gives
$$
\beta\varphi_i\alpha^{-1}=\psi_{\gamma(i)},
$$
i.e.,
\begin{equation}\label{e16}
\varphi_i=\beta^{-1}\psi_{\gamma(i)}\alpha, \ \ \ \ \ \
\varphi_i^{-1}=\alpha^{-1}\psi_{\gamma(i)}^{-1}\beta.
\end{equation}
Thus for $p=\gamma^{-1}(1)$, where $1$ is the identity of $Q(\circ)$, we obtain
$$
\varphi_p\varphi_i^{-1}\varphi_p=
(\beta^{-1}\psi_1\alpha)(\alpha^{-1}\psi_{\gamma(i)}^{-1}\beta)(\beta^{-1}\psi_1\alpha)=
\beta^{-1}(\psi_1\psi_{\gamma(i)}^{-1}\psi_1)\alpha .
$$
Since $Q(\circ)$ is a $D$-loop, for $k=\gamma^{-1}\psi_1\gamma(i)$, by Corollary \ref{C2.3}, we have
$$
\beta^{-1}(\psi_1\psi_{\gamma(i)}^{-1}\psi_1)\alpha=\beta^{-1}\psi_{\gamma(i)^{-1}}\alpha=\beta^{-1}\psi_{\psi_1\gamma(i)}\alpha=\beta^{-1}\psi_{\gamma(k)}\alpha= \varphi_{k} .
$$
So, $\varphi_p\varphi_i^{-1}\varphi_p=\varphi_k$, which means that \eqref{e14}
is valid for $\sigma=\gamma^{-1}\psi_1\gamma$.
\end{proof}

The converse statement is more complicated.

\begin{theorem}\label{T3.8}
Let a quasigroup $Q(\cdot)$ and a loop $Q(\circ)$ with the identity $1$ be isotopic, i.e., let \eqref{e15} holds. If $\varphi_i$ are tracks of $Q(\cdot)$, $\psi_i$ -- tracks of $Q(\circ)$ and \eqref{e14} is satisfied for $p=\gamma^{-1}(1)$,\ $\sigma=\gamma^{-1}\psi_1\gamma$ and all $i\in Q,$ then $Q(\circ)$ is a $D$-loop.
\end{theorem}
\begin{proof}
Indeed, \eqref{e15} holds, then for $p=\gamma^{-1}(1)$ and any $i\in Q$, in view of \eqref{e16}, we have
$$
\arraycolsep=.5mm\begin{array}{rll}
\psi_1\psi_i^{-1}\psi_1&=
(\beta\varphi_{\gamma^{-1}(1)}\alpha^{-1})(\alpha\varphi_{\gamma^{-1}(i)}^{-1}\beta^{-1})(\beta\varphi_{\gamma^{-1}(1)}\alpha^{-1})
=\beta(\varphi_p\varphi_{\gamma^{-1}(i)}^{-1}\varphi_p)\alpha^{-1}\\[4pt]
&= \beta\varphi_{\sigma(\gamma(i))}\alpha^{-1} =
\beta\varphi_{\gamma^{-1}\varphi_1(i)}\alpha^{-1}=
\psi_{\psi_1(i)}=\psi_{i^{-1}} ,
\end{array}
$$
where $i^{-1}$ is calculated in $Q(\circ)$.

Thus $ \psi_1\psi_i^{-1}\psi_1=\psi_{i^{-1}}$, which
means that $Q(\circ)$ is a D-loop.
\end{proof}

\begin{lemma}\label{L3.9}
A loop $Q(\circ)$ is a principal isotope of a quasigroup
$Q(\cdot)$ if and only if
$$
x\circ y=R_b^{-1}(x)\cdot L_a^{-1}(y),
$$
for some $a,b\in Q$ such that $a\cdot b=1$, where $1$ is the
identity of $Q(\circ)$ and $L_a,R_b$ are translations of $Q(\cdot)$.
\end{lemma}
\begin{proof} Indeed, if $Q(\circ)$ is a loop with the identity $1$ and $x\circ
y=\alpha(x)\cdot\beta(y)$ for some permutations $\alpha$, $\beta$
of $Q$, then for $a=\alpha(1)$, $b=\beta(1)$ we have
\begin{eqnarray*}
&&1=1\circ 1=\alpha(1)\cdot\beta(1)=a\cdot b,\\[3pt]
&&x=x\circ 1=\alpha(x)\cdot\beta(1)=\alpha(x)\cdot b,\\[3pt]
&&y=1\circ y=\alpha(1)\cdot\beta(y)=a\cdot\beta(y).
\end{eqnarray*}
Thus
$$
\alpha(x)=R_b^{-1}(x), \ \ \ \ \beta(y)=L_a^{-1}(y).
$$
Hence $x\circ y= R_b^{-1}(x)\cdot L_a^{-1}(y)$.

The converse statement is obvious.
\end{proof}

\begin{corollary}
A quasigroup $Q(\cdot)$ is a principal isotope of a loop
$Q(\circ)$ with the identity $1$ if and only if
$$
x\cdot y=R_b(x)\circ L_a(y),
$$
for some translations $L_a$, $R_b$ of $Q(\cdot)$ such that $a\cdot
b=1$.\qed
\end{corollary}

\begin{proposition}
In any principal isotope $Q(\cdot)$ of a $D$-loop $Q(\circ)$ with the
identity $1$ we have
$$
\varphi_1\varphi_i^{-1}\varphi_1=\varphi_{i^{-1}},
$$
where $i^{-1}$ is calculated in $Q(\circ)$.
\end{proposition}
\begin{proof}
It is a consequence of \eqref{e15} and \eqref{e16}.
\end{proof}
\begin{corollary}
A principal isotope $Q(\cdot)$ of a $D$-loop $Q(\circ)$ is a $D$-loop if and only if $Q(\cdot)$ and $Q(\circ)$
have the same inverse elements.\hfill$\Box{}$
\end{corollary}

\begin{corollary}
A principal isotope $Q(\cdot)$ of a $D$-loop $Q(\circ)$ is a $D$-loop if and only if
$Q(\cdot)$ and $Q(\circ)$ have the same tracks induced by the identity of $Q(\circ)$, i.e., if and only if $\varphi_1$ and $\psi_1$, where $1$ he identity of $Q(\circ)$. \hfill$\Box{}$
\end{corollary}

\section*{\centerline{4. Proper D-loops}}\setcounter{section}{4}\setcounter{theorem}{0}

A $D$-loop is {\it proper} if it is not an IP-loop. The smallest $D$-loop has six elements. Below we present a full list of all nonisotopic proper $D$-loops of order $6$. They represent (respectively) the classes $8.1.1$, $9.1.1$, $10.1.1$ and $11.1.1$ mentioned in the book \cite{DK}.

\bigskip

\centerline{\small{$
\begin{array}{c|ccccccccccc}
\cdot&1&2&3&4&5&6\\ \hline
1&1&2&3&4&5&6\rule{0pt}{10pt}\\
2&2&1&6&5&3&4\\
3&3&6&1&2&4&5\\
4&4&5&2&1&6&3\\
5&5&3&4&6&1&2\\
6&6&4&5&3&2&1
\end{array}
\rule{15mm}{0mm}
\begin{array}{c|ccccccccccc}
\cdot&1&2&3&4&5&6\\ \hline
1&1&2&3&4&5&6\rule{0pt}{10pt}\\
2&2&3&1&6&4&5\\
3&3&1&2&5&6&4\\
4&4&6&5&1&2&3\\
5&5&4&6&2&3&1\\
6&6&5&4&3&1&2
\end{array}$}}

\bigskip

\centerline{\small{$
\begin{array}{c|ccccccccccc}
\cdot&1&2&3&4&5&6\\ \hline
1&1&2&3&4&5&6\rule{0pt}{10pt}\\
2&2&1&6&5&4&3\\
3&3&5&1&2&6&4\\
4&4&6&2&1&3&5\\
5&5&3&4&6&2&1\\
6&6&4&5&3&1&2
\end{array}
\rule{15mm}{0mm}
\begin{array}{c|ccccccccccc}
\cdot&1&2&3&4&5&6\\ \hline
1&1&2&3&4&5&6\rule{0pt}{10pt}\\
2&2&1&4&5&6&3\\
3&3&4&2&6&1&5\\
4&4&5&6&2&3&1\\
5&5&6&1&3&2&4\\
6&6&3&5&1&4&2
\end{array}$}}

\footnotesize{\rightline{Received \ September 23, 2012}

\noindent
I.I. Deriyenko\\
Department of Higher Mathematics and Informatics,
Kremenchuk State Polytechnic University,\\
20 Pervomayskaya str.,39600 Kremenchuk, Ukraine\\
E-mail: ivan.deriyenko@gmail.com\\

\medskip\noindent
W.A. Dudek\\
Institute of Mathematics and Computer Science, Wroclaw University of Technology,\\ 
Wyb. Wyspia\'nskiego 27,  
50-370 Wroclaw, Poland\\
E-mail: Wieslaw.Dudek@pwr.wroc.pl }
\end{document}